\documentclass[12pt]{amsart}
\usepackage{mathrsfs}
\usepackage{amssymb}
\usepackage{verbatim}
\usepackage{hyperref}
\usepackage{color}
\usepackage{amsfonts}
\usepackage{mathrsfs}
\usepackage{amsmath}
\usepackage{amssymb}
\usepackage{subeqnarray}
\usepackage{cases}

\begin{document}
\newtheorem{theorem}{Theorem}
\newtheorem{proposition}[theorem]{Proposition}
\newtheorem{assumption}[theorem]{Assumption}
\newtheorem{conjecture}[theorem]{Conjecture}
\def\theconjecture{\unskip}
\newtheorem{corollary}[theorem]{Corollary}
\newtheorem{lemma}[theorem]{Lemma}
\newtheorem{sublemma}[theorem]{Sublemma}
\newtheorem{observation}[theorem]{Observation}
\theoremstyle{definition}
\newtheorem{definition}{Definition}
\newtheorem{notation}[definition]{Notation}
\newtheorem{remark}[definition]{Remark}
\newtheorem{question}[definition]{Question}
\newtheorem{questions}[definition]{Questions}
\newtheorem{example}[definition]{Example}
\newtheorem{problem}[definition]{Problem}
\newtheorem{exercise}[definition]{Exercise}

\numberwithin{theorem}{section}
\numberwithin{definition}{section}
\numberwithin{equation}{section}

\def\earrow{{\mathbf e}}
\def\rarrow{{\mathbf r}}
\def\uarrow{{\mathbf u}}
\def\varrow{{\mathbf V}}
\def\tpar{T_{\rm par}}
\def\apar{A_{\rm par}}

\def\reals{{\mathbb R}}
\def\torus{{\mathbb T}}
\def\heis{{\mathbb H}}
\def\integers{{\mathbb Z}}
\def\naturals{{\mathbb N}}
\def\complex{{\mathbb C}\/}
\def\distance{\operatorname{distance}\,}
\def\support{\operatorname{support}\,}
\def\dist{\operatorname{dist}\,}
\def\Span{\operatorname{span}\,}
\def\degree{\operatorname{degree}\,}
\def\kernel{\operatorname{kernel}\,}
\def\dim{\operatorname{dim}\,}
\def\codim{\operatorname{codim}}
\def\trace{\operatorname{trace\,}}
\def\Span{\operatorname{span}\,}
\def\dimension{\operatorname{dimension}\,}
\def\codimension{\operatorname{codimension}\,}
\def\nullspace{\scriptk}
\def\kernel{\operatorname{Ker}}
\def\ZZ{ {\mathbb Z} }
\def\p{\partial}
\def\rp{{ ^{-1} }}
\def\Re{\operatorname{Re\,} }
\def\Im{\operatorname{Im\,} }
\def\ov{\overline}
\def\eps{\varepsilon}
\def\lt{L^2}
\def\diver{\operatorname{div}}
\def\curl{\operatorname{curl}}
\def\etta{\eta}
\newcommand{\norm}[1]{ \|  #1 \|}
\def\expect{\mathbb E}
\def\bull{$\bullet$\ }

\def\xone{x_1}
\def\xtwo{x_2}
\def\xq{x_2+x_1^2}
\newcommand{\abr}[1]{ \langle  #1 \rangle}

\newcommand{\Norm}[1]{ \left\|  #1 \right\| }
\newcommand{\set}[1]{ \left\{ #1 \right\} }
\def\one{\mathbf 1}
\def\whole{\mathbf V}
\newcommand{\modulo}[2]{[#1]_{#2}}
\def \essinf{\mathop{\rm essinf}}
\def\scriptf{{\mathcal F}}
\def\scriptg{{\mathcal G}}
\def\scriptm{{\mathcal M}}
\def\scriptb{{\mathcal B}}
\def\scriptc{{\mathcal C}}
\def\scriptt{{\mathcal T}}
\def\scripti{{\mathcal I}}
\def\scripte{{\mathcal E}}
\def\scriptv{{\mathcal V}}
\def\scriptw{{\mathcal W}}
\def\scriptu{{\mathcal U}}
\def\scriptS{{\mathcal S}}
\def\scripta{{\mathcal A}}
\def\scriptr{{\mathcal R}}
\def\scripto{{\mathcal O}}
\def\scripth{{\mathcal H}}
\def\scriptd{{\mathcal D}}
\def\scriptl{{\mathcal L}}
\def\scriptn{{\mathcal N}}
\def\scriptp{{\mathcal P}}
\def\scriptk{{\mathcal K}}
\def\frakv{{\mathfrak V}}
\def\C{\mathbb{C}}
\def\D{\mathcal{D}}
\def\Dn{\mathcal{D}_n}
\def\Dm{\mathcal{D}_m}
\def\R{\mathbb{R}}
\def\Rn{{\mathbb{R}^n}}
\def\Rm{{\mathbb{R}^m}}
\def\Sn{{{S}^{n-1}}}
\def\M{\mathbb{M}}
\def\N{\mathbb{N}}
\def\Q{{\mathbb{Q}}}
\def\Z{\mathbb{Z}}
\def\F{\mathcal{F}}
\def\L{\mathcal{L}}
\def\S{\mathcal{S}}
\def\supp{\operatorname{supp}}
\def\dist{\operatorname{dist}}
\def\essi{\operatornamewithlimits{ess\,inf}}
\def\esss{\operatornamewithlimits{ess\,sup}}
\author{Mingming Cao}
\address{Mingming Cao \\
         School of Mathematical Sciences \\
         Beijing Normal University \\
         Laboratory of Mathematics and Complex Systems \\
         Ministry of Education \\
         Beijing 100875 \\
         People's Republic of China}
\email{m.cao@mail.bnu.edu.cn}

\author{Qingying Xue}
\address{Qingying Xue\\
        School of Mathematical Sciences\\
        Beijing Normal University \\
        Laboratory of Mathematics and Complex Systems\\
        Ministry of Education\\
        Beijing 100875\\
        People's Republic of China}
\email{qyxue@bnu.edu.cn}

\thanks{The second author was supported partly by NSFC
(No. 11471041), the Fundamental Research Funds for the Central Universities (No. 2014KJJCA10) and NCET-13-0065.\\ \indent
Corresponding
author: Qingying Xue\indent Email: qyxue@bnu.edu.cn}


\date{April 23, 2016.}
\title[Bi-parameter Square Function]{Non-homogeneous $Tb$ Theorem for Bi-parameter $g$-Function}
\maketitle

\begin{abstract}
The main result of this paper is a bi-parameter $Tb$ theorem for Littlewood-Paley $g$-function, where $b$ is a tensor product of two pseudo-accretive functions. Instead of the doubling measure, we work with a product measure $\mu = \mu_n \times \mu_m$, where the measures $\mu_n$ and $\mu_m$ are only assumed to be upper doubling. The main techniques of the proof include a bi-parameter $b$-adapted Haar function decomposition and an averaging identity over good double Whitney regions. Moreover, the non-homogeneous analysis and probabilistic methods are used again.
\end{abstract}
\section{Introduction}
It is well-known that the multi-parameter harmonic analysis originated in the work of Fefferman and Stein \cite{F-S}, where the bi-parameter singular integral operators of convolution type are carefully considered. Before long, Journ\'{e} \cite{J} proved the first multi-parameter $T1$ theorem for product spaces by treating the singular integral operator as a vector-valued one-parameter operator. Recently, a new type of $T1$ theorem on product spaces was formulated by Pott and Villarroya \cite{PV}. The authors avoided the vector-valued assumptions by the mixed type conditions including kernel estimates, BMO, and weak boundedness property. Along this way, Martikainen \cite{M2012} gave a bi-parameter representation of singular integrals by dyadic shifts, which extended the famous one-parameter case of Hyt\"{o}nen \cite{H}. Moreover, by means of probabilistic methods and the techniques of dyadic analysis, Hyt\"{o}nen and Martikainen \cite{HM} showed a bi-parameter $T1$ theorem in spaces of non-homogeneous type. Inspired by this, Ou \cite{Ou-1} obtained a bi-parameter $Tb$ theorem on product Lebesgue spaces, where $b$ is a tensor product of two pseudo-accretive functions. Still more recently, a bi-parameter $T1$ theorem for bi-parameter $g$-function was established by Martikainen \cite{M2014}, although the assumptions imposed on the non-convolution kernels seem to be somewhat complicated. The proof was based on modern dyadic probabilistic techniques adapted to the bi-parameter situation.

This paper is devoted to study the non-homogeneous $Tb$ theorem for bi-parameter Littlewood-Paley $g$-function, which is defined by
\begin{align*}
g(f)(x)
&:= \bigg(\int_{0}^{\infty} \int_{0}^{\infty} |\Theta_{t_1,t_2} f(x_1,x_2)|^2 \frac{dt_1}{t_1} \frac{dt_2}{t_2} \bigg)^{1/2},\ x=(x_1,x_2) \in \R^{n+m},
\end{align*}
where the linear term $\Theta_{t_1,t_2}$ is defined by
$$
\Theta_{t_1,t_2} f(x_1,x_2) = \iint_{\R^{n+m}} K_{t_1,t_2}(x_1,x_2,y_1,y_2)f(y_1,y_2) d\mu_n(y_1) \ d\mu_m(y_2),\ \ t_1,t_2 > 0.
$$
More specifically, we will prove $L^2(\mu)$ boundedness of bi-parameter $g$-function on the product space $\R^{m+n} = \Rn \times \Rm$ equipped with a product measure
$\mu = \mu_n \times \mu_m$, where the measures $\mu_n$ and $\mu_m$ are only assumed to be upper doubling.
We also investigate $Tb$ theorem for bi-parameter $g$-function in this article. In other words, to obtain the $L^2(\mu)$ boundedness criterion for it, we will seek the conditions that the function $b$ satisfies. Indeed, we give a sufficient and necessary condition as follows.
\begin{definition}[Bi-parameter Carleson condition]
Let $\D=\Dn \times \Dm$, where $\mathcal{D}_n$ is a dyadic grid on $\Rn$ and $\Dm$ is a dyadic grid on $\R^m$.
Let $W_I = I \times (\ell(I)/2, \ell(I)]$ be Whitney region associated with $I \in \Dn$. Denote
\begin{align*}
C_{I J}^{b} &= \iint_{W_J} \iint_{W_I} |\Theta_{t_1,t_2}b(x_1, x_2)|^2 d\mu_n(x_1)\frac{dt_1}{t_1} d\mu_m(x_2) \frac{dt_2}{t_2} .
\end{align*}
We say $b$ satisfies the $bi$-$parameter \ Carleson \ condition$: For every $\D$ there holds that
\begin{equation}\label{Car-condition}
\sum_{\substack{I \times J \in \D \\ I \times J \subset \Omega}} C_{I J}^{b} \lesssim \mu(\Omega)
\end{equation}
for all sets $\Omega \subset \R^{n+m}$ such that $\mu(\Omega) < \infty$ and such that for every $x \in \Omega$ there exists
$I \times J \in \D$ so that $x \in I \times J \subset \Omega$.
\end{definition}
Although the testing condition in one-parameter setting \cite{MM} is weaker than the bi-parameter Carleson condition, the latter is convenient to deal with our paraproduct estimate.
Additionally, the necessity of it ensures the reasonableness of this formulation.

Compare with classical methods, whether in the one-parameter case or in the multi-parameter case, the dyadic probabilistic techniques is more powerful. The details are exposed to the recent developments, such as \cite{CLX}, \cite{CX}, \cite{LL}, \cite{M2012}, \cite{Ou-1}, \cite{Ou-2} and \cite{OPS}. It is not only more natural to split the summations or integral regions, but also easier to calculate. Furthermore, together with non-homogeneous analysis pioneered by Nazarov, Treil and Volberg \cite{NTV2003}, the probabilistic methods enable us to improve some doubling theories to the non-doubling situation. We will continue to adopt these techniques to our bi-parameter $g$-function. In addition, we need the bi-parameter $b$-adapted Haar functions. In the one-parameter setting, they were introduced by Hyt\"{o}nen \cite{H-vector}.

\section{Definitions and the main theorem}
In this section, we will introduce the definitions and framework which are necessary for the formulation of our main theorem. First, we consider the following class of measures.
\begin{definition}
\textbf{(Upper doubling measures).}
Let $\lambda : \Rn \times (0,\infty) \rightarrow (0,\infty)$ be a function so that $r \mapsto \lambda(x,r)$ is non-decreasing and
$\lambda(x,2r) \leq C_{\lambda} \lambda(x,r)$ for all $x \in \Rn$ and $r>0$. We say that a Borel measure $\mu$ in $\Rn$ is upper
doubling with the dominating function $\lambda$, if $\mu(B(x,r)) \leq \lambda(x,r)$ for all $x \in \Rn$ and $r > 0$.
We set $d_{\lambda} = \log_2 C_{\lambda}$.
\end{definition}

The property $\lambda(x, |x - y|) \simeq \lambda(y, |x - y|)$ can be assumed without loss of generality. Indeed, in Proposition 1.1 \cite{HYY}, it is shown that $\Lambda(x, r) := \inf_{z \in \Rn} \lambda(z, r + |x-z|)$ satisfies that $r \mapsto \Lambda(x, r)$ is non-decreasing, $\Lambda(x, 2r) \leq C_{\lambda} \Lambda(x, r)$, $\mu(B(x, r)) \leq \Lambda(x, r)$, $\Lambda(x, r) \leq \lambda(x, r)$
and $\Lambda(x, r) \leq C_{\lambda} \Lambda(y, r)$ if $|x - y| \leq r$. Therefore, we may (and do) always assume that dominating functions
$\lambda$ satisfy the additional symmetry property $\lambda(x, r) \leq C \lambda(y, r)$ if $|x - y| \leq r$.

From now on, let $\mu = \mu_n \times \mu_m$, where $\mu_n$ and $\mu_m$ are upper doubling measures on $\Rn$ and $\Rm$ respectively.
The corresponding dominating functions are denoted by $\lambda_n$ and $\lambda_m$. We use, for minor convenience, $\ell^\infty$ metrics on $\Rn$ and $\R^m$.
\begin{definition}
A function $b \in L^\infty(\mu)$ is called pseudo-accretive if there is a positive constant $C$ such that for any rectangle
$R \subset \Rn \times \Rm$ with sides parallel to axes,
$$
\frac{1}{\mu(R)} \bigg| \int_{R} b(x) d\mu(x) \bigg| > C.
$$
\end{definition}
In this paper, we will only discuss the case when $b = b_1 \otimes b_2$, where $b_1$ and $b_2$ are in $L^\infty(\mu_n)$ and $L^\infty(\mu_m)$, respectively. Then, the pseudo-accretivity and boundedness of $b$ imply that there exists a constant $C$ such that for any cubes
$I \subset \Rn,\ J  \subset \Rm$,
$$
\frac{1}{\mu_n(I)} \bigg| \int_{I} b_1 \ d\mu_n \bigg| > C,\ \text{and }
\frac{1}{\mu_m(J)} \bigg| \int_{J} b_2 \ d\mu_m \bigg| > C.
$$
That is, $b_1$ and $b_2$ are both pseudo-accretive in the classical sense.

Next, we introduce some appropriate assumptions on kernels that we need throughout the argument. We always assume that the fixed numbers satisfying $\alpha, \beta>0$.

\vspace{0.3cm}
\noindent\textbf{Assumption 2.3 (Standard estimates).} The kernel $K_{t_1,t_2}: \R^{n+m} \times \R^{n+m} \rightarrow \C$ is assumed to satisfy the following estimates:
\begin{enumerate}
\item [(1)] Size condition :
\begin{align*}
|K_{t_1,t_2}(x,y)|
&\lesssim \frac{t_1^{\alpha}}{t_1^{\alpha} \lambda_n(x_1,t_1)+ |x_1 - y_1|^{\alpha} \lambda_n(x_1,|x_1 - y_1|)} \\
&\quad\times  \frac{t_2^{\beta}}{t_2^{\beta} \lambda_m(x_2,t_2)+ |x_2 - y_2|^{\beta} \lambda_m(x_2,|x_2 - y_2|)}.
\end{align*}
\item [(2)] H\"{o}lder condition :
\begin{align*}
&|K_{t_1,t_2}(x,y) - K_{t_1,t_2}(x,(y_1,y_2')) - K_{t_1,t_2}(x,(y_1',y_2)) + K_{t_1,t_2}(x,y')| \\
&\lesssim \frac{|y_1 - y_1'|^{\alpha}}{t_1^{\alpha} \lambda_n(x_1,t_1)+ |x_1 - y_1|^{\alpha} \lambda_n(x_1,|x_1 - y_1|)} \\
&\quad\times  \frac{|y_2 - y_2'|^{\beta}}{t_2^{\beta} \lambda_m(x_2,t_2)+ |x_2 - y_2|^{\beta} \lambda_m(x_2,|x_2 - y_2|)},
\end{align*}
whenever $|y_1 -y_1'| < t_1/2$ and $|y_2 - y_2'| < t_2/2$.
\item [(3)] Mixed H\"{o}lder and size conditions :
\begin{align*}
|K_{t_1,t_2}(x,y) - K_{t_1,t_2}(x,(y_1,y_2'))|
&\lesssim \frac{t_1^{\alpha}}{t_1^{\alpha} \lambda_n(x_1,t_1)+ |x_1 - y_1|^{\alpha} \lambda_n(x_1,|x_1 - y_1|)} \\
&\quad \times \frac{|y_2 - y_2'|^{\beta}}{t_2^{\beta} \lambda_m(x_2,t_2)+ |x_2 - y_2|^{\beta} \lambda_m(x_2,|x_2 - y_2|)},
\end{align*}
whenever $|y_2 -y_2'| < t_2/2$ and
\begin{align*}
|K_{t_1,t_2}(x,y) - K_{t_1,t_2}(x,(y_1',y_2))|
&\lesssim \frac{|y_1 - y_1'|^{\alpha}}{t_1^{\alpha} \lambda_n(x_1,t_1)+ |x_1 - y_1|^{\alpha} \lambda_n(x_1,|x_1 - y_1|)} \\
&\quad \times \frac{t_2^{\beta}}{t_2^{\beta} \lambda_m(x_2,t_2)+ |x_2 - y_2|^{\beta} \lambda_m(x_2,|x_2 - y_2|)},
\end{align*}
whenever $|y_1 - y_1'| < t_1/2$.
\end{enumerate}
\noindent\textbf{Assumption 2.4 (Carleson condition $\times$ Standard estimates).}
If $I \subset \Rn$ is a cube with side length $\ell(I)$, we define the associated Carleson box by $\widehat{I}=I \times (0,\ell(I)]$. We assume the following conditions : For every cube $I \subset \Rn$ and $J \subset \R^m$, there holds that
\begin{enumerate}
\item [(1)] Mixed Carleson and size conditions :
\begin{align*}
\bigg( \iint_{\widehat{I}} \bigg| \int_{I} b_1(y_1) K_{t_1,t_2}(x,y_1,y_2) d\mu_n(y_1) &\bigg|^2 d\mu_n(x_1)\frac{dt_1}{t_1} \bigg)^{1/2}\\&
\lesssim \frac{t_2^{\beta} \ \mu_n(I)^{1/2}}{t_2^{\beta}\lambda_m(x_2,t_2) + |x_2 - y_2|^\beta \lambda_m(x_2,|x_2 - y_2|)}
\end{align*}
and
\begin{align*}
\bigg( \iint_{\widehat{J}} \bigg| \int_{J} b_2(y_2) K_{t_1,t_2}(x,y_1,y_2) d\mu_m(y_2) &\bigg|^2 d\mu_m(x_2) \frac{dt_2}{t_2} \bigg)^{1/2}\\
&\lesssim  \frac{t_1^{\alpha} \ \mu_m(J)^{1/2}}{t_1^{\alpha} \lambda_n(x_1,t_1) + |x_1 - y_1|^\alpha \lambda_n(x_1,|x_1 - y_1|)}.
\end{align*}
\item [(2)] Mixed Carleson and H\"{o}lder conditions :
\begin{align*}
\bigg( \iint_{\widehat{I}} \bigg| \int_{I}& b_1(y_1)[ K_{t_1,t_2}(x,y_1,y_2) - K_{t_1,t_2}(x,y_1,y_2') ] d\mu_n(y_1) \bigg|^2  d\mu_n(x_1) \frac{dt_1}{t_1} \bigg)^{1/2} \\
&\lesssim \frac{|y_2 - y_2'|^{\beta} \ \mu_n(I)^{1/2}}{t_2^{\beta}\lambda_m(x_2,t_2) + |x_2 - y_2|^\beta \lambda_m(x_2,|x_2 - y_2|)}, \
\text{ whenever } |y_2 - y_2'| < t_2/2.
\end{align*}
And
\begin{align*}
\bigg( \iint_{\widehat{J}} \bigg| \int_{J}& b_2(y_2)[ K_{t_1,t_2}(x,y_1,y_2) - K_{t_1,t_2}(x,y_1',y_2) ] d\mu_m(y_2) \bigg|^2  d\mu_m(x_2)\frac{dt_2}{t_2} \bigg)^{1/2} \\
&\lesssim \frac{|y_1 - y_1'|^{\alpha} \ \mu_m(J)^{1/2}}{t_1^{\alpha} \lambda_n(x_1,t_1) + |x_1 - y_1|^\alpha \lambda_n(x_1,|x_1 - y_1|)}, \
\text{whenever } |y_1 - y_1'| < t_1/2.
\end{align*}
\end{enumerate}
We can now formulate our main theorem.
\begin{theorem}\label{Theorem}
Let $\mu=\mu_n \times \mu_m$, where $\mu_n$ and $\mu_m$ are upper doubling measures on $\Rn$ and $\Rm$ respectively. Let $b$ be a pseudo-accretive function on $\Rn \times \Rm$. Assume that the kernels $\{K_{t_1,t_2}\}$ satisfy the Assumptions $2.3-2.4$. If the function $b$ satisfies bi-parameter Carleson condition, then there holds that
\begin{equation}\label{L^2}
\big\| g(f) \big\|_{L^2(\mu)} \lesssim \big\| f \big\|_{L^2(\mu)}.
\end{equation}
Additionally, the bi-parameter Carleson condition is necessary in the following sense :
$K_{t_1,t_2} = K_{t_1} \otimes K_{t_2}$ and the one-parameter kernels satisfy the size condition and corresponding square function bounds.
\end{theorem}

As for the proof of the necessity in our main Theorem $\ref{Theorem}$, we follow exactly the same scheme of proof of the necessity \cite{M2014} with slight modifications.
Some non-homogeneous arguments can be adapted from Lemma 8.7-8.9 \cite{HM}.
Moreover, an important tool is Journ\'{e}¡¯s covering lemma with general product measures, which was given in Theorem 8.1 \cite{HM}. We omit the details.

\section{Preliminaries}
In this section, our goal is to introduce some fundamental tools including the random dyadic grids, good/bad cubes, and $b$-adapted Haar functions. Based on these, we give some reductions of the initial estimate.
\subsection{Random Dyadic Grids}
We here will introduce the fundamental technique, random dyadic grids. Let $\beta_n = \{ \beta_n^j\}_{j \in \Z}$, where
$\beta_n^j \in \{0,1\}^n$. Let $\mathcal{D}_n^0$ be the standard dyadic grids on $\Rn$. In $\Rn$, we define the new dyadic
grid
$$ \mathcal{D}_n = \Big\{I + \beta_n; I \in \mathcal{D}_n^0 \Big\} := \Big\{I + \sum_{j:2^{-j}<\ell(I)} 2^{-j} \beta_n^j; I \in \mathcal{D}_n^0 \Big\}.$$
The dyadic grid $\mathcal{D}_m$ in $\R^m$ is similarly defined. There is a natural product probability structure on
$(\{0, 1\}^n)^{\Z}$ and $(\{0, 1\}^m)^{\Z}$. So we have independent random dyadic grids $\mathcal{D}_n$ and $\mathcal{D}_m$ in $\Rn$ and $\R^m$ respectively. Even if $n = m$ we need two independent grids.

\begin{definition}
A cube $I \in \mathcal{D}_n$ is said to be $bad$ if there exists a $J \in \mathcal{D}_n$ with $\ell(J) \geq 2^r \ell(I)$ such that
$\dist(I,\partial J) \leq \ell(I)^{\gamma_n} \ell(J)^{1-\gamma_n}$. Otherwise, $I$ is called $good$. Here $r \in \Z_+$ and $\gamma_n \in (0,\frac12)$ are given parameters.
\end{definition}
Denote $\pi_{good}^n = \mathbb{P}_{\beta_n}(I + \beta_n \ \text{is \ good}) = \mathbb{E}_{\beta_n}(\mathbf{1}_{good}(I+\beta_n))$.
Then $\pi_{good}^n$ is independent of $I \in \mathcal{D}_n^0$, and the parameter $r$ is a fixed constant so that $\pi_{good}^n,\pi_{good}^m > 0$.
Throughout this article, we take $\gamma_n = \frac{\alpha}{2(d_{\lambda_n}+\alpha)}$, where $\alpha > 0$ appears in the kernel estimates.
It is important to observe that the position and goodness of a cube $I \in \mathcal{D}_n^0$ are independent.

\subsection{$b$-adapted Haar functions}
The abbreviation $b_1(E):=\int_E b_1 \ d\mu_n$ will be used.
For each $I \in \Dn$, we denote its dyadic children by $I_1,\ldots,I_{2^n}$. We index $\{I_j\}$ in such a way that
$$
|b_1(I^*_j)| \geq [1-(k-1)2^{-n}] \mu_n(I),\ \
I^*_{j} = \bigcup_{k=j}^{2^n} I_k,\ j=1,\ldots,2^n.
$$
The existence of such way was shown in Lemma 4.2 \cite{H-vector}. The $b_1$-adapted Haar function is defined by
$$
\varphi_{I,j}^{b_1} :=
\bigg( \frac{b_1(I_j) b_1(I^*_{j+1})}{b_1(I^*_j)}\bigg)^{1/2}
\bigg( \frac{\mathbf{1}_{I_j}}{b_1(I_j)} - \frac{\mathbf{1}_{I^*_{j+1}}}{b_1(I^*_{j+1})}\bigg).
$$
Similarly, we can define the function $\psi_{J,k}^{b_2}$ with respect to $b_2$ and $J \in \D_m$.

The adapted Haar functions enjoy the following properties :
\begin{enumerate}
\item $\int_{\Rn} b_1 \varphi_{I,j}^{b_1} \ d\mu_n = 0$.
\item $|\varphi_{I,j}^{b_1}| \simeq \mu_n(I_j)^{1/2} \bigg( \frac{\mathbf{1}_{I_j}}{b_1(I_j)} + \frac{\mathbf{1}_{I^*_{j+1}}}{b_1(I^*_{j+1})}\bigg)$.
\item $\big\| \varphi_{I,j}^{b_1} \big\|_{L^p(\mu_n)} \simeq \mu_n(I_j)^{1/p-1/2}$, $p \in [1,\infty]$.
\item The similar above properties hold for $\psi_{J,k}^{b_2}$ as well.
\item For any $f \in L^2(\mu)$, there holds that
$$
f = \sum_{j=1}^{2^n} \sum_{k=1}^{2^m} \sum_{I \in \Dn} \sum_{J \in \Dm} \langle f, \varphi_{I,j}^{b_1} \otimes \psi_{J,k}^{b_2} \rangle \
b \cdot \varphi_{I,j}^{b_1} \otimes \psi_{J,k}^{b_2}.
$$
\end{enumerate}
The properties (1)-(4) can be found in Proposition 4.3 \cite{H-vector}. Property (5) can be verified by iteration of the one-parameter argument.
\subsection{Averaging over Good Whitney Regions}
Let $f \in L^2(\mu)$. Let always $I_1, I_2 \in \mathcal{D}_n$ and $J_1, J_2 \in \mathcal{D}_m$.
Note that the position and goodness of $I + \beta_n$ are independent. Therefore, one can write
$$
\big\| g(f) \big\|_{L^2(\mu)}^2
= c_{m,n} \mathbb{E}_{\beta_n} \mathbb{E}_{\beta_m} \Sigma_{\beta_n,\beta_m},
$$
where $c_{m,n}=(\pi_{good}^n \cdot \pi_{good}^m)^{-1}$ and
$$
\Sigma_{\beta_n,\beta_m}:= \sum_{I_2,J_2: good} \iint_{W_{J_2}} \iint_{W_{I_2}}
\big| \Theta_{t_1,t_2} f(x) \big|^2 d\mu_n(x_1) \frac{dt_1}{t_1} d\mu_m(x_2) \frac{dt_2}{t_2}.
$$
Indeed, to get this equality, we only need to apply the similar argument to one-parameter case twice. For more details in one-parameter setting, see \cite{CLX}.
Then, applying $b$-adapted Haar decomposition of $f$ (suppressing the finite $j,k$ summation), we may further write
$$
\Sigma_{\beta_n,\beta_m}= \sum_{I_2,J_2: good} \iint_{W_{J_2}} \iint_{W_{I_2}}
\Big| \sum_{I_1,J_1}  f_{I_1 J_1} \Theta_{t_1,t_2} (b \cdot \varphi^{b_1}_{I_1} \otimes \psi^{b_2}_{J_1})(x)\Big|^2 d\mu_n(x_1) \frac{dt_1}{t_1} d\mu_m(x_2) \frac{dt_2}{t_2}.
$$
When $\beta_n$ and $\beta_m$ are fixed, we denote $\Sigma_{\beta_n,\beta_m}$ by $\Sigma$. Consequently, it is enough to show $\Sigma \lesssim ||f||_{L^2(\mu)}^2$, where the implied constant is independent of $\beta_n$ and $\beta_m$.

We can preform the decomposition
$$ \Sigma \lesssim \Sigma_{<,<} + \Sigma_{<,\geq} + \Sigma_{\geq,<} + \Sigma_{\geq,\geq},$$
where
\begin{align*}
\Sigma_{<,<}:= &\sum_{I_2,J_2: good} \iint_{W_{J_2}} \iint_{W_{I_2}} \Big| \sum_{\substack{I_1,J_1 \\ \ell(I_1) < \ell(I_2) \\ \ell(J_1) < \ell(J_2)}}
 f_{I_1 J_1} \Theta_{t_1,t_2} (b \cdot \varphi^{b_1}_{I_1} \otimes \psi^{b_2}_{J_1})(x) \Big|^2 d\mu_n(x_1) \frac{dt_1}{t_1} d\mu_m(x_2) \frac{dt_2}{t_2},
\end{align*}
and the others are completely similar.

Sequentially, it suffices to focus on controlling the four pieces: $\Sigma_{<,<}$, $\Sigma_{<,\geq}$, $\Sigma_{\geq,<}$, $\Sigma_{\geq,\geq}$ in the following sections.
\section{Some standard estimates}
This section is devoted to proving some estimates, which will be used at certain points in our proof.
\begin{lemma}\label{Lemma-1}
Assume that $I_1,I_2 \in \Dn$ with $\ell(I_1) < \ell(I_2)$. Denote
$$
\mathscr{F}_{n,\alpha}(I_1,x_1,t_1) :=
\int_{I_1} \frac{|y_1 - c_{I_1}|^\alpha}{t_1^{\alpha} \lambda_n(x_1,t_1)+ |x_1 - y_1|^{\alpha} \lambda_n(x_1,|x_1 - y_1|)} d\mu_n(y_1).
$$
Then there holds that
$$
\mathscr{F}_{n,\alpha}(I_1,x_1,t_1)
\lesssim A_{I_1 I_2} \mu_n(I_1)^{-1/2} \mu_n(I_2)^{-1/2}.
$$
\end{lemma}
\begin{proof}
We begin by the estimate
$$
\mathscr{F}_{n,\alpha}(I_1,x_1,t_1)
\lesssim \frac{\ell(I_1)^\alpha}{\ell(I_2)^{\alpha} \lambda_n(x_1,\ell(I_2))+ d(I_1,I_2)^{\alpha} \lambda_n(x_1,d(I_1,I_2))}.
$$
Note that if $\ell(I_2) \leq d(I_1,I_2)$, then $D(I_1,I_2) \simeq d(I_1,I_2)$. If $\ell(I_2) > d(I_1,I_2)$, then $D(I_1,I_2) \simeq \ell(I_2)$.
Moreover, for any $z_1 \in I_1 \cup I_2$, it holds that $|x_1 - z_1| \lesssim D(I_1,I_2)$, which implies that $\lambda_n(z_1,D(I_1,I_2)) \simeq \lambda_n(x_1,D(I_1,I_2))$.
Hence, we have
\begin{align*}
\mathscr{F}_{n,\alpha}(I_1,x_1,t_1)
\lesssim \frac{\ell(I_1)^{\alpha/2} \ell(I_2)^{\alpha/2}}{D(I_1,I_2) \sup\limits_{z_1 \in I_1 \cup I_2} \lambda_n(z_1,D(I_1,I_2))}
= A_{I_1 I_2} \mu_n(I_1)^{-1/2} \mu_m(I_2)^{-1/2}.
\end{align*}
\end{proof}
\begin{lemma}\label{Lemma-2}
Let $k \geq 1$ and $I \in \Dn$ be a good cube. Set
$$
\mathfrak{F}_k(x_1):=\int_{(I^{(k-1)})^c} \frac{t_1^\alpha}{t_1^{\alpha} \lambda_n(x_1,t_1)+ |x_1 - y_1|^{\alpha} \lambda_n(x_1,|x_1 - y_1|)} d\mu_n(y_1).
$$
Then we have the geometric decay
$\mathfrak{F}_k(x_1) \lesssim 2^{-\alpha k/2}$.
\end{lemma}
\begin{proof}
If $k \leq r$, we get
\begin{align*}
\mathfrak{F}_k(x_1)
&\lesssim \frac{\mu_n(3I)}{\lambda_n(x_1,\ell(I))} + \ell(I)^{\alpha} \int_{(3I)^c} \frac{|y_1 - c_I|^{-\alpha}}{\lambda_n(c_I,|y_1 - c_I|)} d\mu_n(y_1)
\lesssim 1 \simeq 2^{-\alpha k/2}.
\end{align*}
If $k > r$, we have by the goodness of $I$ that
$$ d(I,(I^{(k-1)})^c) > \ell(I)^{\gamma_n} \ell(I^{(k-1)})^{1-\gamma_n}= 2^{(k-1)(1-\gamma_n)} \ell(I) \gtrsim 2^{k/2} \ell(I).$$
Thus, it immediately yields that
\begin{align*}
\mathfrak{F}_k(x_1)
&\leq  \ell(I)^{\alpha} \int_{B(x_1,d(I,(I^{(k-1)})^c))} \frac{|y_1 - x_1|^{-\alpha}}{\lambda_n(x_1,|y_1 - x_1|)} d\mu_n(y_1)\\&
\lesssim \ell(I)^\alpha d(I,(I^{(k-1)})^c)^{-\alpha} \lesssim 2^{-\alpha k/2}.
\end{align*}
\end{proof}
We need the following lemma, which can be found in \cite{NTV2003}.
\begin{lemma}(\cite{NTV2003}) \label{A-I1-I2}
Let us set
$$
A_{I_1 I_2}=\frac{\ell(I_1)^{\alpha/2} \ell(I_2)^{\alpha/2}}{D(I_1,I_2)^{\alpha} \sup\limits_{z_1 \in I_1 \cup I_2} \lambda_n(z_1,D(I_1,I_2))} \mu_n(I_1)^{1/2} \mu_n(I_2)^{1/2},
$$
where $D(I_1,I_2)=\ell(I_1) + \ell(I_2) + d(I_1,I_2)$, $I_1,I_2 \in \mathcal{D}_n$. Then for any $x_{I_1}, y_{I_2} \geq 0$, we have the following estimate
$$
\Big( \sum_{I_1,I_2} A_{I_1 I_2} x_{I_1} y_{I_2} \Big)^2 \lesssim \sum_{I_1} x_{I_1}^2
\times \sum_{I_2} y_{I_2}^2.
$$
In particular, there holds that
$$ \sum_{I_2} \Big( \sum_{I_1} A_{I_1 I_2} x_{I_1} \Big)^2 \lesssim \sum_{I_1}x_{I_1}^2. $$
\end{lemma}
\begin{lemma}\label{Lemma-4}
If we denote
$$
a_I := \iint_{W_{I}} \Big|\Theta_{t_1,t_2} (b_1 \otimes (b_2 \psi^{b_2}_{J_1}))(x) \Big|^2 d\mu_n(x_1) \frac{dt_1}{t_1},
$$
then $\{ a_I \}_{I \in \Dn}$ is a Carleson sequence. Rather, there holds for any $I \in \Dn$
\begin{equation}\label{Car-1}
\sum_{I': I' \subset I} a_{I'} \lesssim \big(A_{J_1 J_2} \mu_m(J_2)^{-1/2}\big)^2 \mu_n(I).
\end{equation}
\end{lemma}
\begin{proof}
We are in the position of showing $\{ a_I \}_{I \in \Dn}$ is a Carleson sequence. Indeed, we have
\begin{align*}
\sum_{I': I' \subset I} a_{I'}
&=\iint_{\widehat{I}} \Big|\Theta_{t_1,t_2} (b_1 \otimes (b_2 \psi^{b_2}_{J_1}))(x) \Big|^2 d\mu_n(x_1) \frac{dt_1}{t_1} \\
&\leq \iint_{\widehat{3I}} \Big|\Theta_{t_1,t_2} ((b_1 \mathbf{1}_{3I}) \otimes (b_2 \psi^{b_2}_{J_1}))(x) \Big|^2 d\mu_n(x_1) \frac{dt_1}{t_1} \\
&\quad + \iint_{\widehat{I}} \Big|\Theta_{t_1,t_2} ((b_1 \mathbf{1}_{(3I)^c}) \otimes (b_2 \psi^{b_2}_{J_1}))(x) \Big|^2 d\mu_n(x_1) \frac{dt_1}{t_1} \\
&:= \mathfrak{T}_1 + \mathfrak{T}_2.
\end{align*}
Combining cancellation property, the Minkowski inequality, with the mixed Carleson and the H\"{o}lder condition, we get
\begin{align*}
\mathfrak{T}_1^{1/2}
& \leq \int_{J_1} |\psi^{b_2}_{J_1}(y_2)|\bigg( \iint_{\widehat{3I}} \bigg| \int_{3I} b_1(y_1) \big[K_{t_1,t_2}(x,y) \\&\quad- K_{t_1,y_2}(x,(y_1,c_{J_1})) \big] d\mu_m(y_1) \bigg|^2 d\mu_n(x_1)\frac{dt_1}{t_1} \bigg)^{1/2} d\mu_m(y_2) \\
&\lesssim \mu_m(J_1)^{-1/2} \mu_n(I)^{1/2} \mathscr{F}_{m,\beta}(J_1,x_2,t_2) \\
&\lesssim A_{J_1 J_2} \mu_m(J_2)^{-1/2} \mu_n(I)^{1/2}.
\end{align*}
As for the second part, the size condition implies that
\begin{align*}
\big|\Theta_{t_1,t_2} ((b_1 \mathbf{1}_{(3I)^c}) \otimes (b_2 \psi^{b_2}_{J_1}))(x) \big|
& \lesssim  t_1^\alpha \int_{(3I)^c} \frac{|y_1- x_1|^{-\alpha}}{\lambda_n(x_1,|y_1 - x_1|)} d\mu_n(y_1) \cdot A_{J_1 J_2} \mu_m(J_2)^{-1/2} \\
& \lesssim t_1^\alpha \int_{(3I)^c} \frac{|y_1- c_I|^{-\alpha}}{\lambda_n(x_1,|y_1 - c_I|)} d\mu_n(y_1) \cdot A_{J_1 J_2} \mu_m(J_2)^{-1/2} \\
& \lesssim  t_1^\alpha \ell(I)^{-\alpha} A_{J_1 J_2} \mu_m(J_2)^{-1/2}.
\end{align*}
This indicates that
$$
\mathfrak{T}_2
\lesssim \big(A_{J_1 J_2} \mu_m(J_2)^{-1/2}\big)^2 \mu_n(I) \cdot \ell(I)^{-2\alpha} \int_{0}^{\ell(I)} t_1^{2 \alpha} \frac{dt_1}{t_1}
\lesssim \big(A_{J_1 J_2} \mu_m(J_2)^{-1/2}\big)^2 \mu_n(I).
$$
Therefore, one obtains the desired result $(\ref{Car-1})$.

\end{proof}
\begin{lemma}\label{Lemma-5}
Let $k \geq 1$ and $I \in \Dn$ be a good cube. We have the following Carleson estimate :
\begin{equation}\label{Car-2}
\sum_{J' : J' \subset J} a_{J'} \lesssim 2^{-\alpha k} \mu_n(I^{(k)})^{-1} \mu_m(J) .
\end{equation}
where
$$
a_J := \iint_{W_{J}}
\big| \Theta_{t_1,t_2}((b_1 \xi_I^k) \otimes b_2)(x) \big|^2 d\mu_m(x_2) \frac{dt_2}{t_2}.
$$
\end{lemma}
\begin{proof}Note that
\begin{align*}
\sum_{J': J' \subset J} a_{J'}
&=\iint_{\widehat{J}} \big| \Theta_{t_1,t_2}((b_1 \xi_I^k) \otimes b_2)(x) \big|^2 d\mu_m(x_2) \frac{dt_2}{t_2} \\
&\leq \iint_{\widehat{3J}} \big| \Theta_{t_1,t_2}((b_1 \xi_I^k) \otimes (b_2 \mathbf{1}_{3J})(x) \big|^2 d\mu_m(x_2) \frac{dt_2}{t_2} \\
&\quad + \iint_{\widehat{J}} \big| \Theta_{t_1,t_2}((b_1 \xi_I^k) \otimes (b_2 \mathbf{1}_{(3J)^c})(x) \big|^2 d\mu_m(x_2) \frac{dt_2}{t_2} \\
&:= \mathfrak{R}_1 + \mathfrak{R}_2.
\end{align*}
The mixed Carleson and size condition yield that
\begin{align*}
\mathfrak{R}_1^{1/2}
&\leq \int_{(I^{(k-1)})^c} |\xi_I^k(y_1)| \bigg(\iint_{\widehat{3J}} \bigg| \int_{3J} b_2(y_2) K_{t_1,t_2}(x,y) d\mu_m(y_2) \bigg|^2 d\mu_n(x_1) dt_1 \bigg)^{1/2} d\mu_n(y_1)\\
&\lesssim \mu_n(I^{(k)})^{-1/2} \mu_m(J)^{1/2} \mathfrak{F}_k(x_1)
\lesssim 2^{-\alpha k/2}\mu_n(I^{(k)})^{-1/2} \mu_m(J)^{1/2}.
\end{align*}
An easy consequence of the size condition is that
\begin{align*}
\big| \Theta_{t_1,t_2}((b_1 \xi_I^k) \otimes (b_2 \mathbf{1}_{(3J)^c})(x) \big|
&\lesssim \mu_n(I^{(k)})^{-1/2} \mathfrak{F}_k(x_1) \cdot t_2^\beta \int_{(3J)^c} \frac{|y_2- x_2|^{-\beta}}{\lambda_m(x_2,|y_2 - x_2|)} d\mu_m(y_2) \\
&\lesssim 2^{-\alpha k/2}\mu_n(I^{(k)})^{-1/2} \cdot t_2^\beta \int_{J^c} \frac{|y_2- c_J|^{-\beta}}{\lambda_m(x_2,|y_2 - c_J|)} d\mu_m(y_2) \\
&\lesssim 2^{-\alpha k/2}\mu_n(I^{(k)})^{-1/2} \cdot t_2^\beta \ell(J)^{-\beta}.
\end{align*}
It immediately lead to the following estimate:
$$
\mathfrak{R}_2 \lesssim 2^{-\alpha k}\mu_n(I^{(k)})^{-1} \mu_m(J).
$$
Hence, the inequality $(\ref{Car-2})$ has been proved.
\end{proof}
Finally, we present a dyadic Carleson embedding theorem, which was proved in \cite{NTV1997}.
\begin{lemma}[Dyadic Carleson Embedding Theorem]\label{Carleson theorem}
If the numbers $a_Q \geq 0$, $Q \in \D$, satisfy the following Carleson measure condition
$$
\sum_{Q' \subset Q} a_{Q'} \leq \nu(Q), \ \text{for each } Q \in \D,
$$
then for any $f \in L^2(\nu)$
$$
\sum_{Q \in \D} a_Q | \langle f \rangle_Q^{\nu}|^2 \leq 4 ||f||_{L^2(\nu)}.
$$
\end{lemma}
\section{The Case : $\ell(I_1) < \ell(I_2)$ and $\ell(J_1) < \ell(J_2)$}
Using the cancellation properties of the adapted Haar functions
$$\int_{\Rn} b_1 \varphi_{I_1}^{b_1} d\mu_n= \int_{\Rm} b_2 \psi_{J_1}^{b_2} d\mu_m= 0,$$ we can replace $K_{t_1,t_2}(x,y)$ by
$$
K_{t_1,t_2}(x,y) - K_{t_1,t_2}(x,(y_1,c_{J_1})) - K_{t_1,t_2}(x,(c_{I_1},y_2)) + K_{t_1,t_2}(x,(c_{I_1},c_{J_1})).
$$
By the full H\"{o}lder condition of the kernel $K_{t_1,t_2}$ and Lemma $\ref{Lemma-1}$, we have
\begin{align*}
\big|\Theta_{t_1,t_2} (b \cdot \varphi^{b_1}_{I_1} \otimes \psi^{b_2}_{J_1})(x)\big|
&\lesssim ||b||_{L^{\infty}(\mu)} ||\varphi^{b_1}_{I_1}||_{L^{\infty}(\mu_n)} ||\psi^{b_2}_{J_1}||_{L^{\infty}(\mu_m)}
\mathscr{F}_{n,\alpha}(I_1,x_1,t_1) \mathscr{F}_{m,\beta}(J_1,x_2,t_2) \\
&\lesssim A_{I_1 I_2} \mu_n(I_2)^{-1/2} \cdot A_{J_1 J_2} \mu_m(J_2)^{-1/2}.
\end{align*}
Therefore, from Lemma $\ref{A-I1-I2}$, it follows that
\begin{align*}
\Sigma_{<,<}
&\lesssim \sum_{J_2} \sum_{I_2} \Big( \sum_{I_1} A_{I_1 I_2} \sum_{J_1} A_{J_1 J_2} |f_{I_1 J_1}| \Big)^2 \\
&\lesssim \sum_{J_2} \sum_{I_1} \Big( \sum_{J_1} A_{J_1 J_2} |f_{I_1 J_1}|  \Big)^2 \\
&\lesssim \sum_{I_1} \sum_{J_1} |f_{I_1 J_1}|^2
\simeq \big\| f \big\|_{L^2(\mu)}^2 .
\end{align*}
\section{The Case : $\ell(I_1) \geq \ell(I_2)$ and $\ell(J_1) < \ell(J_2)$}
Since the mixed H\"{o}lder and size conditions and the mixed of Carleson and H\"{o}lder estimates are symmetric, the control of $\Sigma_{<,\geq}$ is completely symmetric with $\Sigma_{\geq,<}$. Thus we only focus on the domination of $\Sigma_{\geq,<}$.

In any case, we can carry out the splitting
$$ \sum_{\ell(I_1) \geq \ell(I_2)}
=\sum_{\substack{\ell(I_1) \geq \ell(I_2) \\ d(I_1,I_2) > \ell(I_2)^{\gamma_n} \ell(I_1)^{1-\gamma_n}}}
+\sum_{\substack{\ell(I_1) > 2^r \ell(I_2)\\d(I_1,I_2) \leq \ell(I_2)^{\gamma_n} \ell(I_1)^{1-\gamma_n}}}
+\sum_{\substack{\ell(I_2) \leq \ell(I_1) \leq 2^r \ell(I_2) \\ d(I_1,I_2) \leq \ell(I_2)^{\gamma_n}\ell(I_1)^{1-\gamma_n}}}. $$
By restricting $\Sigma_{\geq,<}$ to the above three summation conditions, we obtain corresponding three terms $\Sigma_{out,<}$, $\Sigma_{in}$ and $\Sigma_{near,<}$ respectively.
Thus, there holds
$$
\Sigma_{\geq,<} \lesssim \Sigma_{out,<} + \Sigma_{in,<} + \Sigma_{near,<}.
$$

We next shall treat the above three terms respectively.
\subsection{Separated $\Sigma_{out,<}$.}
We first claim that it must have in this case
\begin{equation}\label{d-D}
\frac{\ell(I_2)^\alpha}{d(I_1,I_2)^{\alpha} \lambda_n(x_1,d(I_1,I_2))}
\lesssim \frac{\ell(I_1)^{\alpha/2} \ell(I_2)^{\alpha/2}}{D(I_1,I_2)^{\alpha} \lambda_n(x_1,D(I_1,I_2))}.
\end{equation}
Indeed, if $d(I_1,I_2) \geq \ell(I_1)$, then $D(I_1,I_2) \simeq d(I_1,I_2)$. So, the inequality $(\ref{d-D})$ holds.
If $d(I_1,I_2) < \ell(I_1)$, then $D(I_1,I_2) \simeq \ell(I_1)$. The doubling condition of $\lambda_n(x_1,t)$ implies that
\begin{align*}
\lambda_n(x_1,\ell(I_1))
&=\lambda_n \big(x_1,(\ell(I_1)/\ell(I_2))^{\gamma_n} \ell(I_2)^{\gamma_n} \ell(I_1)^{1-\gamma_n} \big) \\
&\lesssim C_{\lambda_n}^{\log_2 (\ell(I_1)/\ell(I_2))^{\gamma_n}} \lambda_{n}\big(x_1,\ell(I_2)^{\gamma_n} \ell(I_1)^{1-\gamma_n} \big) \\
&= (\ell(I_1)/\ell(I_2))^{\gamma_n d_n} \lambda_{n}\big(x_1,\ell(I_2)^{\gamma_n} \ell(I_1)^{1-\gamma_n} \big).
\end{align*}
It is important to notice that $\gamma_n (d_{\lambda_n} + \alpha) = \alpha/2$ and $d(I_1,I_2) > \ell(I_2)^{\gamma_n} \ell(I_1)^{1-\gamma_n}$. Hence, one may conclude that
\begin{equation*}
\frac{\ell(I_2)^\alpha}{d(I_1,I_2)^{\alpha} \lambda_n(x_1,d(I_1,I_2))}
\lesssim \frac{\ell(I_2)^{\alpha/2}}{\ell(I_1)^{\alpha/2} \lambda_n(x_1,\ell(I_1))}
\lesssim \frac{\ell(I_1)^{\alpha/2} \ell(I_2)^{\alpha/2}}{D(I_1,I_2)^{\alpha} \lambda_n(x_1,D(I_1,I_2))}.
\end{equation*}
This shows the inequality $(\ref{d-D})$.

We continue the proof. By the cancellation property, we can change the kernel $K_{t_1,t_2}(x,y)$ to
$$
K_{t_1,t_2}(x,y)-K_{t_1,t_2}(x,(y_1,c_{J_1})).
$$
The mixed H\"{o}lder and size condition implies that
\begin{eqnarray}\label{Theta-b}
\big|\Theta_{t_1,t_2} (b \cdot \varphi^{b_1}_{I_1} \otimes \psi^{b_2}_{J_1})(x)\big|
&{}& \nonumber \lesssim \mu_n(I_1)^{-1/2} \mu_m(J_1)^{-1/2} \mathscr{F}_{m,\beta}(J_1,x_2,t_2) \\
&{}& \nonumber \ \times \int_{I_1} \frac{t_1^\alpha}{t_1^{\alpha} \lambda_n(x_1,t_1)+ |x_1 - y_1|^{\alpha} \lambda_n(x_1,|x_1 - y_1|)} d\mu_n(y_1) \\
&{}& \lesssim \frac{\ell(I_2)^\alpha \ \mu_n(I_1)^{1/2}}{\ell(I_2)^{\alpha} \lambda_n(x_1,\ell(I_2))+ d(I_1,I_2)^{\alpha} \lambda_n(x_1,d(I_1,I_2))}  A_{J_1 J_2} \mu_m(J_2)^{-1/2} \\
&{}& \nonumber  \lesssim  A_{I_1 I_2} \mu_n(I_2)^{-1/2} \cdot A_{J_1 J_2} \mu_m(J_2)^{-1/2}.
\end{eqnarray}
Accordingly, applying the similar argument as $\Sigma_{<,<}$, we have
$$
\Sigma_{sep,<} \lesssim \big\| f \big\|_{L^2(\mu)}^2.
$$
\subsection{Nearby $\Sigma_{near,<}$.}
The summation conditions $\ell(I_2) \leq \ell(I_1) \leq 2^r \ell(I_2)$ and $d(I_1,I_2) \leq \ell(I_2)^{\gamma_n} \ell(I_1)^{1- \gamma_n}$ indicate that $\ell(I_1) \simeq \ell(I_2) \simeq D(I_1,I_2)$. Thus, for convenience, we write $I_1 \simeq I_2$ in this case. There holds that
\begin{equation}\label{Sim-Sim}
\frac{\mu_n(I_1)^{1/2}}{\lambda_n(x_1,\ell(I_2))}
\simeq \frac{\mu_n(I_1)^{1/2}}{\lambda_n(c_{I_1},\ell(I_1))^{1/2}} \lambda_n(c_{I_2},\ell(I_2))^{-1/2}
\leq \mu_n(I_2)^{-1/2}.
\end{equation}
It follows from $(\ref{Theta-b})$ and $(\ref{Sim-Sim})$ that
\begin{align*}
\big|\Theta_{t_1,t_2} (b \cdot \varphi^{b_1}_{I_1} \otimes \psi^{b_2}_{J_1})(x)\big|
\lesssim  \mu_n(I_2)^{-1/2} \cdot A_{J_1 J_2} \mu_m(J_2)^{-\frac12}.
\end{align*}
Notice that for a given $I_2$, there are finite cubes $I_1$ such that $I_1 \simeq I_2$. The same conclusion holds for a given $I_1$.
Therefore, we obtain that
\begin{align*}
\Sigma_{near,<}
&\lesssim \sum_{I_2} \sum_{J_2} \sum_{I_1:I_1 \simeq I_2} \Big(\sum_{J_1} A_{J_1 J_2} |f_{I_1 J_1}| \Big)^2  \\
&\lesssim \sum_{I_1} \sum_{J_2} \Big( \sum_{J_1} A_{J_1 J_2} |f_{I_1 J_1}|  \Big)^2 \sum_{I_2: I_2 \simeq I_1} 1\\
&\lesssim \sum_{I_1} \sum_{J_1} |f_{I_1 J_1}|^2
\simeq \big\| f \big\|_{L^2(\mu)}^2 .
\end{align*}
\subsection{Inside $\Sigma_{in,<}$.}
In this case, by the goodness of $I_2$, it must actually have $I_2 \subsetneq I_1$. We use $I^{(k)} \in \mathcal{D}_n$ to denote the unique cube for which $\ell(I^{(k)}) = 2^k \ell(I)$ and $I \subset I^{(k)}$. This enables us to write
\begin{align*}
\Sigma_{in,<}
= \sum_{I,J_2:good} \iint_{W_{J_2}} \iint_{W_{I}} \Big| \sum_{k=1}^\infty
\sum_{J_1:\ell(J_1)<\ell(J_2)} f_{I^{(k)} J_1} \Theta_{t_1,t_2} (b \cdot \varphi^{b_1}_{I^{(k)}} \otimes \psi^{b_2}_{J_1})(x) \Big|^2
d\mu_n(x_1) \frac{dt_1}{t_1} d\mu_m(x_2) \frac{dt_2}{t_2}.
\end{align*}

Introduce the notation
\begin{equation}\label{xi-I-k}
\xi_I^k = - \langle \varphi^{b_1}_{I^{(k)}} \rangle_{I^{(k-1)}} \mathbf{1}_{(I^{(k-1)})^c}
+ \sum_{\substack{I' \in ch(I^{(k)}) \\ I' \neq I^{(k-1)}}} \varphi^{b_1}_{I^{(k)}} \mathbf{1}_{I'}.
\end{equation}
It is easy to check that $\supp \xi_I^k \subset (I^{(k-1)})^c$, $||\xi_I^k||_{L^{\infty}(\mu_n)} \lesssim \mu_n(I^{(k)})^{-1/2}$, and
\begin{equation}\label{h-I-k}
\varphi^{b_1}_{I^{(k)}} = \xi_I^k + \langle \varphi^{b_1}_{I^{(k)}} \rangle_{I^{(k-1)}}.
\end{equation}
Denote $f_{J_1} = \langle f,\psi^{b_2}_{J_1} \rangle$ so that $f_{J_1}(y_1) = \int_{\R^m} f(y_1,y_2) \psi^{b_2}_{J_1}(y_2) d\mu_m(y_2)$, $y_1 \in \Rn$.

We then split
$$ \Sigma_{in,<} \lesssim \Sigma_{mod,<} + \Sigma_{Car,<} \ \ ,$$
where
\begin{align*}
\Sigma_{mod,<}
&= \sum_{I,J_2: good} \iint_{W_{J_2}} \iint_{W_{I}} \Big| \sum_{k=1}^\infty \sum_{J_1:\ell(J_1)<\ell(J_2)}
f_{I^{(k)} J_1} \Theta_{t_1,t_2} (b \cdot \xi_I^k \otimes \psi^{b_2}_{J_1})(x) \Big|^2 d\mu_n(x_1) \frac{dt_1}{t_1} d\mu_m(x_2) \frac{dt_2}{t_2}
\end{align*}
and
\begin{align*}
\Sigma_{Car,<}
= \sum_{I,J_2: good} \iint_{W_{J_2}} \iint_{W_{I}} \Big| \sum_{J_1:\ell(J_1)<\ell(J_2)} \langle f_{J_1} \rangle_{I} \Theta_{t_1,t_2} (b_1 \otimes (b_2 \psi^{b_2}_{J_1}))(x) \Big|^2 d\mu_n(x_1) \frac{dt_1}{t_1} d\mu_m(x_2) \frac{dt_2}{t_2}.
\end{align*}
The term $\Sigma_{Car,<}$ is obtained by the fact
$$
\langle f_{J_1} \rangle_{I}
= \sum_{k=1}^\infty f_{I^{(k)} J_1} \langle \varphi^{b_1}_{I^{(k)}} \rangle_{I^{(k-1)}}.
$$
\vspace{0.3cm}
\noindent\textbf{$\bullet$ Estimate of $\Sigma_{mod,<}$.}
Using the cancellation property again, we can change the kernel to $K_{t_1,t_2}(x,y)-K_{t_1,t_2}(x,(y_1,c_{J_1}))$. The mixed H\"{o}lder and size condition gives that
\begin{align*}
\big|\Theta_{t_1,t_2} (b \cdot \xi_I^k \otimes \psi^{b_2}_{J_1})(x) \big|
&\lesssim  \mu_n(I^{(k-1)})^{-1/2} \mathfrak{F}_k(x_1) \cdot A_{J_1 J_2} \mu_m(J_2)^{-1/2} \\
&\lesssim 2^{-\alpha k/2} \mu_n(I^{(k-1)})^{-1/2} \cdot A_{J_1 J_2} \mu_m(J_2)^{-1/2},
\end{align*}
provided by Lemma $\ref{Lemma-2}$.
Accordingly, from Minkowski's integral inequality and H\"{o}lder's inequality, we conclude that
\begin{align*}
\Sigma_{mod,<}
&\lesssim \sum_{I} \sum_{J_2} \bigg( \sum_{k=1}^\infty 2^{-\alpha k/2} \Big( \frac{\mu_n(I)}{\mu_n(I^{(k)})} \sum_{J_1:\ell(J_1) < \ell(J_2)} A_{J_1 J_2} |f_{I^{(k)} J_1}| \Big)^{1/2} \bigg)^2 \\
&\leq \bigg[ \sum_{k=1}^\infty 2^{-\alpha k/4} \cdot 2^{-\alpha k/4}\bigg( \sum_I \frac{\mu_n(I)}{\mu_n(I^{(k)})} \sum_{J_2} \Big( \sum_{J_1:\ell(J_1) < \ell(J_2)} A_{J_1 J_2} |f_{I^{(k)} J_1}| \Big)^2 \bigg)^{1/2} \bigg]^2 \\
&\lesssim \sum_{k=1}^\infty 2^{-\alpha k/2} \sum_I \frac{\mu_n(I)}{\mu_n(I^{(k)})} \sum_{J_2} \Big( \sum_{J_1:\ell(J_1) < \ell(J_2)} A_{J_1 J_2} |f_{I^{(k)} J_1}| \Big)^2 \\
&\lesssim \sum_{k=1}^\infty 2^{-\alpha k/2} \sum_{Q,J_1} |f_{Q J_1}|^2 \mu_n(Q)^{-1} \sum_{I:I^{(k)}=Q}\mu_n(I)
\lesssim \big\| f \big\|_{L^2(\mu)}^2 .
\end{align*}
\qed

\noindent\textbf{$\bullet$ Estimate of $\Sigma_{Car,<}$.}
By Minkowski's inequality, we have
$$
\mathcal{G}_{Car,<}
\leq \sum_{J_2} \iint_{W_{J_2}} \bigg(\sum_{J_1:\ell(J_1)<\ell(J_2)} \Big(\sum_{I} |\langle f_{J_1} \rangle_{I}|^2 a_I \Big)^{1/2} \bigg)^2 d\mu_m(x_2) \frac{dt_2}{t_2},
$$
where $a_I$ is defined as Lemma $\ref{Lemma-4}$.
Therefore, by Lemma $\ref{Lemma-4}$ and Carleson embedding theorem, we obtain the following estimate
\begin{align*}
\Sigma_{Car,<}
&\lesssim \sum_{J_2} \Big( \sum_{J_1} A_{J_1 J_2} \big\| f_{J_1} \big\|_{L^2(\mu_n)} \Big)^2
\lesssim \sum_{J_1} \big\| f_{J_1} \big\|_{L^2(\mu_n)}^2 \lesssim \big\| f \big\|_{L^2(\mu)}^2.
\end{align*}
So far, we have completed the estimate of $\Sigma_{\geq,<}$.

\qed
\section{The Case : $\ell(I_1) \geq \ell(I_2)$ and $\ell(J_1) \geq \ell(J_2)$.}
As we did above, the summation $\ell(I_1) \geq \ell(I_2)$ was split into three pieces.
A similar decomposition in the summation $\ell(J_1) \geq \ell(J_2)$ can be also performed. This leads to
\begin{align*}
\Sigma_{\geq,\geq}
&\lesssim \Sigma_{out,out} + \Sigma_{out,in} + \Sigma_{out,near} + \Sigma_{in,out} + \Sigma_{in,in} \\
&\quad  + \Sigma_{in,near} + \Sigma_{near,out}  + \Sigma_{near,in} + \Sigma_{near,near}.
\end{align*}
\subsection{Nested/Nested : $\Sigma_{out,out}$.}
We begin with the term $\Sigma_{out,out}$, where the new bi-parameter phenomena will appear. Using the similar decomposition to $(\ref{h-I-k})$, we can split the function
$\psi_{J}^{b_2}$ with $\eta_{J}^{i}$. Thus, it holds that
$$\Sigma_{nes,nes}
\lesssim \Sigma_{mod,mod} + \Sigma_{Car,Car} + \Sigma_{mod,Car} + \Sigma_{Car,mod} \ ,$$
where
\begin{align*}
\Sigma_{mod,mod}
&= \sum_{I,J: good} \iint_{W_{J}} \iint_{W_{I}} \Big| \sum_{k=1}^\infty \sum_{i=1}^\infty
f_{I^{(k)} J^{(i)}} \Theta_{t_1,t_2}(b \cdot \xi_I^k \otimes \eta_J^{i})(x) \Big|^2 d\mu_n(x_1) \frac{dt_1}{t_1} d\mu_m(x_2) \frac{dt_2}{t_2},
\end{align*}
\begin{align*}
\Sigma_{mod,Car}
&= \sum_{I,J: good} \iint_{W_{J}} \iint_{W_{I}} \Big| \sum_{k=1}^\infty
\langle f_{I^{(k)}} \rangle_J \Theta_{t_1,t_2}((b_1 \xi_I^k) \otimes b_2)(x) \Big|^2 d\mu_n(x_1) \frac{dt_1}{t_1} d\mu_m(x_2) \frac{dt_2}{t_2},
\end{align*}
\begin{align*}
\Sigma_{Car,mod}
&= \sum_{I,J: good} \iint_{W_{J}} \iint_{W_{I}} \Big| \sum_{i=1}^\infty
\langle f_{J^{(i)}} \rangle_I \Theta_{t_1,t_2}(b_1 \otimes (b_2 \eta_J^i))(x) \Big|^2  d\mu_n(x_1) \frac{dt_1}{t_1} d\mu_m(x_2) \frac{dt_2}{t_2},
\end{align*}
and
\begin{align*}
\Sigma_{Car,Car}
&= \sum_{I,J: good} |\langle f \rangle_{I \times J}|^2 \iint_{W_{J}} \iint_{W_{I}} \iint_{\R^{n+m}} | \Theta_{t_1,t_2} b(x) |^2 d\mu_n(x_1) \frac{dt_1}{t_1}
d\mu_m(x_2) \frac{dt_2}{t_2}.
\end{align*}
\subsubsection{\bf{Estimate of} $\Sigma_{mod,mod}$.}
The size condition leads to the bound
$$
\big|\Theta_{t_1,t_2}(b \cdot \xi_I^k \otimes \eta_J^{i})(x)\big|
\lesssim ||\xi_I^k||_{L^{\infty}(\mu_n)} ||\eta_J^{i}||_{L^{\infty}(\mu_m)} \mathfrak{F}_k(x_1) \mathfrak{G}_i(x_2),
$$
where
$$
\mathfrak{G}_i(x_2)
:= \int_{(J^{(i-1)})^c} \frac{t_2^{\beta}}{t_2^{\beta} \lambda_m(x_2,t_2)+ |x_2 - y_2|^{\beta} \lambda_m(x_2,|x_2 - y_2|)} d\mu_m(y_2).
$$
Using the standard estimates as $\mathfrak{F}_k(x_1)$, we can get
$$
\mathfrak{G}_i(x_2) \lesssim 2^{-\beta i}.
$$
Therefore, there holds
\begin{align*}
\big|\Theta_{t_1,t_2}(b \cdot \xi_I^k \otimes \eta_J^{i})(x)\big|
\lesssim 2^{-\alpha k/2} \mu_n(I^{(k)})^{-1/2} \cdot 2^{-\beta i} \mu_m(J^{(i)})^{-1/2}.
\end{align*}
Applying the similar estimate to $\Sigma_{mod,<}$ to analyze $\Sigma_{mod,mod}$, we deduce that
\begin{align*}
\Sigma_{mod,mod}
&\lesssim \sum_{k,i} 2^{-\alpha k/2} 2^{-\beta i} \sum_{Q,R}|f_{QR}|^2 \frac{1}{|Q|} \sum_{I:I^{(k)}=Q} |I| \cdot \frac{1}{|R|} \sum_{J:J^{(i)}=R} |J|
\lesssim \big\| f \big\|_{L^2(\mu)}^2 .
\end{align*}
\qed
\subsubsection{\bf{Estimate of} $\Sigma_{mod,Car}$ and $\Sigma_{Car,mod}$.}
Now we turn our attention to dominate $\Sigma_{mod,Car}$. From the Minkowski inequality, it follows that
$$
\Sigma_{mod,Car}
\leq \sum_{I: good} \iint_{W_{I}} \bigg[ \sum_{k=1}^\infty \bigg( \sum_{J} | \langle f_{I^{(k)}} \rangle_J |^2 a_J \bigg)^{1/2} \bigg]^2 d\mu_n(x_1) \frac{dt_1}{t_1}.
$$
Then Carleson embedding theorem indicates that
$$
\Sigma_{mod,Car}
\lesssim \sum_{i=1}^\infty 2^{-\beta i/2} \sum_{R} \big\| f_R \big\|_{L^2(\mu_n)}^2 \frac{1}{|R|} \sum_{J:J^{(i)=R}}|I|
\lesssim \big\| f \big\|_{L^2(\mu)}^2 .
$$
\qed
\subsubsection{\bf{Estimate of} $\Sigma_{Car,Car}$.}
Applying the bi-parameter Carleson condition, it immediately yields that
\begin{align*}
\Sigma_{Car,Car}
&=\sum_{I,J}|\langle f \rangle_{I \times J}|^2 C_{I J}^b
= 2 \int_{0}^{\infty} \sum_{\substack{I,J \\ |\langle f \rangle_{I \times J}| > t}} C_{I J}^{b}\ t \ dt \\
&\lesssim \int_{0}^{\infty} \sum_{\substack{I,J \\ I \times J \subset \{M_s^{\mathcal{D}} f > t\}}} C_{I J}^{b}\ t \ dt
\lesssim \int_{0}^{\infty} \mu \big(\{M_s^{\mathcal{D}} f > t\} \big) t \ dt \\
&\lesssim \big\| M_s^{\mathcal{D}} f \big\|_{L^2(\mu)}^2
\lesssim \big\| f \big\|_{L^2(\mu)}^2,
\end{align*}
where in the last step we have used the $L^p(\mu)(1<p<\infty)$ boundedness of the strong maximal function associated with rectangles.
\qed
\subsection{The rest of terms.}
As for the estimates of the remaining terms, they are simply combinations of the techniques we have used above. Thereby, we here only present certain key points.

Applying the size condition or the mixed H\"{o}lder and size condition, we can dominate the terms $\Sigma_{out,out}$, $\Sigma_{out,near}$, $\Sigma_{near,near}$ and $\Sigma_{near,out}$ directly. For the terms $\Sigma_{in,out}$ and $\Sigma_{in,near}$, $nes$, they can be split into $mod$ and $Car$. To bound them, it suffices to use the size condition and the combinations of Carlson and size estimate. The terms $\Sigma_{out,in}$ and $\Sigma_{near,in}$ are symmetric with respect to them, respectively.

\qed

\end{document}